\documentclass[11pt,twoside]{article} 
\usepackage[margin=1.2in]{geometry}
\usepackage{amsmath}
\usepackage{amsthm} 
\usepackage{graphicx}
\graphicspath{{./images/}}
\newtheorem{lemma}{Lemma}

\newtheorem{proposition}{Proposition}
\newtheorem{theorem}{Theorem}

\usepackage{caption} 
\usepackage{subcaption} 
\usepackage{amsfonts} 
\usepackage{enumitem}

\title{\textbf{
A fast algorithm for computing Bell polynomials based on index break-downs using prime factorization
}}
\author{Hamed Taghavian \\
\small{KTH Royal Institute of Technology} \\
\small{hamedta@kth.se}}
\begin{document}
\maketitle

\abstract{
By establishing an interesting connection between ordinary Bell polynomials and rational convolution powers, several properties of Bell polynomials are obtained including some composition and inverse relations. Based on these results, new algorithms are presented for calculation of partial Bell polynomials. It is shown that the proposed method is more efficient than the conventional recurrence procedure for computing these polynomials, requiring far less arithmetic operations in most cases. A detailed analysis of the computation complexity is provided, followed by some numerical evaluations.\\
\textbf{Key words:} \textit{Bell polynomials; Primes; Convolution; Inverse relation; Computation complexity.}
}

\section{Introduction}\label{sec:intro}
Bell polynomials play an essential role in many basic mathematical problems appearing in nonlinear algebraic equations, nonlinear differential equations, integro-differential equations, probability theory, matrix algebra, combinatorics and umbral calculus. Partial exponential Bell polynomials are defined by \cite[p.412]{10}

\begin{equation} \label{3}
B_{n,k} (x)=\sum \binom{n}{j_1,j_2 \ldots j_n} \left(\frac{x_1}{1!}\right)^{j_1} \left(\frac{x_2}{2!}\right)^{j_2} \cdots \left(\frac{x_n}{n!}\right)^{j_n}   
\end{equation} 
based on the sequence $x =\lbrace x_1, x_2, \ldots \rbrace$, where the summation is over all $n$-tuples $j$ satisfying $j_1 ,j_2 ,\ldots ,j_n \in {\mathbb{Z}}^{\ge 0}$, $j_1 +j_2 +\ldots +j_n =k$ and $j_1 +2j_2 +\ldots +nj_n =n$. For (\ref{3}) to be well-defined, the length of sequence $x$ should be greater than or equal to $n-k+1$. There is a tremendous literature on Bell polynomials and their properties because of their remarkable applications, ranging from statistics and number theory to mathematical physics and dynamical systems (see for example \cite{b4}, \cite{bb4}, \cite{b2}, \cite{bb2} \cite{bb3} and \cite{b3}).
As a few simple samples, we may refer to:
\begin{itemize}
\item the relation between determinant of a matrix and its traces:
$$ \textnormal{det}(A)=\frac{1}{n!}B_n(\textnormal{tr}(A),-1!\textnormal{tr}(A^2),\ldots,(-1)^{n-1}(n-1)!\textnormal{tr}(A^n))$$
\item the relation between different symmetric polynomials:
$$p_n=n(-1)^n\sum_{k=1}^n \frac{1}{k}\hat{B}_{n,k}(-e_1,-e_2,\ldots,-e_{n-k+1})$$
where $p$ and $e$ denote power sums and elementary symmetric polynomials respectively.
\item derivatives of composite functions:
$$ \frac{d^n}{dx^n} f(g(x))=\sum_{k=1}^n f^{(k)}(g(x))B_{n,k}(g'(x),g''(x),\ldots,g^{(n-k+1)(x)})$$
\end{itemize}
and so on.

It is acknowledged that Bell polynomials also come useful in convolution calculus \cite{a1}, \cite{a2} and \cite{a3}. The explicit role of these polynomials in that area is described in this paper. Particularly, it is established how rational convolution powers of general sequences are expressed by means of partial ordinary Bell polynomials. An important consequence of these expressions is in the problem of calculating convolution roots of sequences. Problems of this kind were considered in \cite{c1}, \cite{c2}, \cite{c3} and \cite{c4} and in a few special cases were solved in \cite{18} using generating functions. The explicit solution to the most general case of this problem is given in this note. What then follows are some simple curious identities and inverse relations for Bell polynomials.

Inverse relations involving partial Bell polynomials, have been the concern of several recent studies (see for example \cite{a2}, \cite{a4}, \cite{7}, \cite{a5} and \cite{a7}). The obtained relations in these papers share the common trait of incorporating partial Bell polynomials in a summation over the second index. Inverse relations of this kind include the inverse relation of a complete Bell polynomial in \cite{a4} and the inverse relation \cite{a7}

\begin{equation}\label{0}
\left\lbrace \begin{array}{l}
y_n=\sum_{k=1}^n \binom{-an-b}{k}\hat{B}_{n,k}(x) /{(an+b)} \\
x_n=\sum_{k=1}^n \binom{-(an+1)/b}{k}b^k \hat{B}_{n,k}(y) /{(an+1)}
\end{array} \right.
\end{equation}
where $\hat{B}_{n,k}(.)$ denotes partial ordinary Bell polynomials. Apparently however, the explicit inverse of a single partial Bell polynomial does not exist in the literature. This can be seen as the solution to the problem of recovering $x$ given the sequence $y_n=B_{n,k}(x)$ for a fixed $k$, which is addressed in Theorem~\ref{tinv} of the present paper.

The composition identities obtained in this paper on the other hand, provide a way to express consecutive applications of arbitrarily many partial ordinary Bell polynomials on a sequence, as a single Bell polynomial. This result paves the way for a new recursive algorithm (described in (\ref{92})) which can be used to compute partial Bell polynomials $B_{n,k}(.)$ more efficiently. As the provided simulations suggest, this algorithm is capable of decreasing the number of arithmetic operations needed for computing Bell polynomials using the classical recursive approach (by over 99\% in some cases). Combining this with another lemma concerning Bell polynomials with argument sequences having initially zero samples yields in the most efficient Algorithm~3 (\ref{GenAl}) presented in Section~\ref{subsec:first-index} of this paper.

The proposed algorithms can make the scientific computations involving Bell polynomials much less costly. For some explicit samples of computations which involve Bell polynomials, one may refer to computing stability regions and tuning compensators in control systems \cite{b3,IRBell}, signals in electrical circuits \cite{z11, z12}, solutions to relaxation and diffusion processes \cite{a3}, random sums, Gibbs partitions, Moments and cumulants in probability theory \cite{z13, z14}. As a more specific example, computing the compound distribution of the random variable $S_k=\sum_{\tau=1}^{k} X_{\tau}$ is an important numerical problem in insurance mathematics \cite{z15}, which is just given by
$$ \mathbb{P}(S_k=n)=\hat{B}_{n,k} (p)$$
where $p_{n}$, $n\in {\mathbb{Z} ^{\geq 1}}$ is the probability distribution of the \emph{i.i.d} random variables $X_{\tau}$. Moreover due to the far-reaching applications of Bell polynomials, the contributions of this paper can also make the computations more efficient in different special functions \cite{ADD} and sequences connected with Bell polynomials, including Stirling numbers~\cite{ADD2}.

This paper is organized as follows:
\begin{itemize}[noitemsep]
\item The prerequisites and the necessary notation clarifications are provided in Section~\ref{sec:pre}.
\item The composition and inverse relations of Bell polynomials are presented in Section~\ref{sec:relations}.
\item The existing algorithms for computing Bell polynomials are reviewed and new algorithms are obtained for this purpose in Section~\ref{sec:AL}. Analysis of the algorithms, comparisons of the computation costs and numerical experiments are provided in the same section.
\item Concluding remarks are given in Section~\ref{sec:conc}.
\end{itemize}

\section{Preliminaries}\label{sec:pre}
The notation and preliminary information used in the paper are presented in this section. We begin with recalling definition of the general binomial coefficients:

\begin{equation}\label{5.1}
\binom{\alpha}{k}=\frac{\Gamma(\alpha+1)}{\Gamma(\alpha-k+1)\Gamma(k+1)}
\end{equation}
where $\Gamma (.)$ denotes the Gamma function \cite[p.1]{30}. Likewise, the $k$-permutations of $\alpha$ is given by
$$
(\alpha)_k= \frac{\Gamma(\alpha+1)}{\Gamma(\alpha-k+1)}
$$
Bell polynomials (\ref{3}) can be computed using the following recurrence relation \cite[p.415]{10}:

\begin{equation}\label{3.1}
B_{n,k}(x)=\sum_{j=1}^{n-k+1}\binom{n-1}{j-1}x_j B_{n-j,k-1}(x)
\end{equation}
with the initial condition $B_{0,0}(x)=1$ and the convention $B_{n,0}(x)=B_{0,k}(x)=0$ for $n,k \geq 1$. Based on (\ref{3}), partial ordinary Bell polynomials are defined as follows:

\begin{equation} \label{9)}
 \hat{B}_{n,k} (x ) = \frac{k!}{n!} B_{n,k} (y)
\end{equation}
where $y_n=n!x_n$. These polynomials obviously satisfy

\begin{equation}\label{9.1}
\hat{B}_{n,k} (cx)=c^{k} \hat{B}_{n,k} (x )
\end{equation} 
for any real scalar $c \neq 0$, and they can be calculated using the following recurrence relation \cite[p.366]{34}:

\begin{equation} \label{14)} 
\hat{B}_{n,k} (x )=\sum _{j=1}^{n-k+1}x_{j} \hat{B}_{n-j,k-1} (x)  
\end{equation} 
which holds for $n,k\ge 1$ on the condition $\hat{B}_{n,0} (x)=\delta _{n}$, where $\delta$ denotes the Kronecker delta. Moreover partial ordinary Bell polynomials (\ref{9)}) have the generating function \cite{a7}:

\begin{equation}\label{9.2}
\left(\sum_{j=1}^{+\infty } x_j z^{j}\right)^k=\sum_{n=k}^{+\infty} \hat{B}_{n,k} (x) z^{n}
\end{equation}
All discussions in this paper are made for real functions with integer support, $x_{n} :{\mathbb{Z}}\to {\mathbb{R}}$. This paper only focuses on ``right hand'' functions that are zero in the range $\lbrace -\infty ,\cdots ,n_{0} \rbrace$ for some $n_{0} \in {\mathbb{Z}}$. Also the terms ``functions with integer arguments'' and ``sequences'' are used interchangeably, as they both refer to the same concept.
The convolution of two real functions $x_{n}$ and $y_{n}$ with an integer argument $n\in {\mathbb{Z}}$ is defined as

\begin{equation} \label{1)} 
(x*y)_n=\sum _{j=-\infty }^{+\infty }x_j y_{n-j}   
\end{equation}
Accordingly, we denote the convolution power of sequence $x$ as $x^{*k} \buildrel\Delta\over= \overbrace{x*x*\cdots *x}^{k{\rm \; times}}$ where $x_{n}^{*0} =\delta _{n} $ and $x_{n}^{*1} =x_{n}$. A convolution power need not be a whole number and can be extended to rational numbers. For example, ${(1/k)}^{\textnormal{th}}$ convolution power of sequence $x$, denoted by $x^{*1/k}$, is defined as the appropriate sequence that satisfies

\begin{equation} \label{5)} 
(x^{*1/k} )_{n}^{*k} =x_{n}  
\end{equation}
where $k\in {\mathbb{N}}$ and $x$ is a known sequence. It can be shown that the sequence $x^{*1/k}$ in \eqref{5)} only exists if $k|n_{0}$, in which $n_{0}$ is the least integer satisfying $x_{n_{0}} \ne 0$ \cite{18}. In addition, since only a real-valued solution is in the scope of the present paper, it can be deduced from \cite{18} that equation \eqref{5)} has a unique solution for $x^{*1/k}$ if $k$ is odd and the double solution $\pm x^{*1/k}$ otherwise. In \eqref{5)} the function $x^{*1/k}$ is called the $k^{\textnormal{th}}$ convolution root of $x$ and sometimes, it is called a convolution radical.

\section{Composition and inverse relations}\label{sec:relations}
This section concerns with some properties of Bell polynomials which help to establish the computation algorithms in the next section. Convolution calculus can facilitate derivation of these properties, shortening their proofs. Hence we will first describe the connection between Bell polynomials and rational convolution powers. We begin with integer convolution powers which are trivially expressed in terms of Bell polynomials, as shown in the following lemma.

\begin{lemma}\label{l2}
Let $x_{n}$, $n\in {\mathbb{Z}}$ be a sequence satisfying $x_{n} =0$ for $n<1$. The convolution power $x_{n}^{*k} $, $n\in {\mathbb{Z}}$, $k\in {\mathbb{Z}}^{\ge 0}$ is then given by

\begin{equation} \label{13)} 
x_{n}^{*k} =\hat{B}_{n,k} (x) 
\end{equation}
\end{lemma}

\begin{proof}
Noting the fact that $\hat{B}_{n,k} (x)=0$ holds for $n<k$, the proof follows from (\ref{9.2}) in a straight-forward way.
\end{proof}

Using multiple time shifts and the time invariance property of the convolution operator \cite[p.14]{22}, Lemma~\ref{l2} can be easily generalized as

\begin{equation} \label{16)} 
x_{n}^{*k} =\hat{B}_{n+k(1-n_{0}),k} (y) 
\end{equation}
where $y_n=   x_{  n-1+n_{0} }   $ and $x$ is a sequence satisfying $x_{n}=0$ for $n<n_{0}$.

Now consider convolution roots which appear in the problem of solving \eqref{5)}. Assume that $x_{n} =0$ holds for $n<0$ and let $x_{0} \ne 0$. Furthermore, define a rational power of the complex number $z\in {\mathbb{C}}$ as $z^{1/k} =\exp \left(\ln \left(z\right)/k\right)$ in which $k\in {\mathbb{N}}$ and $\ln z$ is a branch of the complex logarithm that is holomorphic on an open disk with the radius $\left|x_{0} \right|$ and center $x_{0}$. A standard approach toward solving \eqref{5)} starts from the following identity based on generating functions:

\begin{equation} \label{48)} 
\sum _{n=0}^{+\infty }x_{n}^{*1/k} z^{n}  =\left(\sum _{n=0}^{+\infty}x_n z^n \right)^{1/k}  
\end{equation} 
Since $f(z)=\sum _{n=0}^{+\infty }x_{n} z^{n}$ is analytic at $z=0$ and $\ln z$ is analytic at $f(0)=x_{0}$, the function appeared on the right side of \eqref{48)} is analytic at $z=0$. This implies that a convergent solution $x_{n}^{*1/k}$ for all $n\in {\mathbb{Z}}$ exists for this problem, which is found in the next lemma.

\begin{lemma}\label{l6}
Let $x_{n}$, $n\in {\mathbb{Z}}$ be a sequence satisfying $x_{n}=0$ for $n<0$ and let $x_{0} \ne 0$. The $k^{\textnormal{th}} $ convolution root $x_{n}^{*1/k} $, $n\in {\mathbb{Z}}$, $k\in {\mathbb{N}}$ is given by

\begin{equation} \label{49)} 
x_{n}^{*1/k} =x_{0} ^{1/k} \sum _{j=0}^{n} \binom{1/k}{j} \hat{B}_{n,j} (x /x_0) 
\end{equation} 
\end{lemma}

\begin{proof}
According to \eqref{48)} it is deduced that:

\begin{equation} \label{50)}
\begin{array}{ll}
x_{n}^{*1/k}				&=\frac{1}{n!} \frac{d^{n} }{dz^{n} } \left(\sum _{n=0}^{+\infty }x_{n} z^{n}  \right)^{1/k} _{|z=0}  \\
							&=\frac{1}{n!} \frac{d^{n} }{dz^{n} } \left(x_{0} +\sum _{n=1}^{+\infty }x_{n} z^{n}  \right)^{1/k} _{|z=0}
\end{array}
\end{equation} 
Since function $g(z)=\sum _{n=1}^{+\infty }x_{n} z^{n}$ is continuous at $z=0$ and $x_{0} \ne 0$, one may write $\left|x_{0} \right|>\left|\sum _{n=1}^{+\infty }x_{n} z^{n}  \right| $ in a neighborhood of the origin $\left|z\right|<\varepsilon$ for some $\varepsilon >0$. Therefore, by using the Newton's generalized binomial Theorem~\cite[p.397]{17} in this region, it is possible to write

\begin{equation} \label{52)} 
x_{n}^{*1/k} =\frac{1}{n!} \sum _{j=0}^{+\infty}    \binom{1/k}{j}   x_{0} ^{1/k-j}  \frac{d^{n} }{dz^{n} } \left(\sum _{n=1}^{+\infty }x_{n} z^{n}  \right)^{j} _{|z=0}  
\end{equation} 
Representation $\hat{B}_{n,j} (x)=\frac{1}{n!} \frac{d^{n} }{dz^{n} } \left(\sum _{n=1}^{+\infty }x_{n} z^{n}  \right)^{j} _{|z=0}$ which can be deduced from the generating function (\ref{9.2}) may then be used to write \eqref{52)} in terms of partial ordinary Bell polynomials as follows:

\begin{equation} \label{53)} 
x_{n}^{*1/k} =\sum _{j=0}^{+\infty }   \binom{1/k}{j}   x_{0} ^{1/k-j} \hat{B}_{n,j} (x)  
\end{equation} 

Since $\hat{B}_{n,j} (x)=0$ holds for $n<j$, only a finite number of terms in the summation \eqref{53)} needs to be considered. This lemma is eventually proved by using the property (\ref{9.1}).
\end{proof}

Lemma~\ref{l6} can be alternatively proved by calculating the successive derivatives of the function $\left(X(z)/x_0\right)^{1/k}$ with respect to $z$ using Theorem~B of \cite[section 3.5]{comtet} and the relation of potential polynomials with Bell polynomials. It can be shown that Lemma~\ref{l6} can be extended to find convolution roots of sequences that are zero in the range $n\in \lbrace -\infty ,\cdots ,n_0 -1 \rbrace$ where $n_{0} \in {\mathbb{Z}}$, $k|n_{0}$ as

\begin{equation} \label{54)} 
x_{n}^{*1/k} =x_{n_{0}} ^{1/k} \sum _{j=0}^{n-n_{0} /k}     \binom{1/k}{j}    \hat{B}_{n-n_{0} /k,j} ( y ) 
\end{equation}
where $ y_n=x_{n+n_0} /x_{n_0} $ and $x_{n}$, $n\in {\mathbb{Z}}$ is a sequence satisfying $x_{n}=0$ for $n<n_{0}$ and $x_{n_{0} } \ne 0$. It shall not be omitted that in case $k$ is even, one may define the convolution root as the expressions obtained above multiplied by $-1$ instead.\\

Based on the above results, some useful identities for Bell polynomials are derived using natural properties of the convolution operator.

\begin{theorem}\label{t1}
There hold the composition relations:

\begin{equation}\label{54.1}
\left\lbrace
\begin{array}{l}
y_n=\hat{B}_{n,k_1}(x) \\
\hat{B}_{n,k_1k_2}(x)=\hat{B}_{n,k_2}(y)
\end{array}
\right.
\end{equation}
and

\begin{equation}\label{54.05}
\left\lbrace
\begin{array}{l}
y_n=B_{n,k_1}(x) \\
B_{n,k_1k_2}(x)=\frac{(k_1!)^{k_2}k_2!}{(k_1k_2)!}B_{n,k_2}(y)
\end{array}
\right.
\end{equation}
\end{theorem}

\begin{proof}
Firstly note that relation $x^{*k_1k_2}=({x}^{*k_1})^{*k_2}$ clearly holds noting the generating functions. Using Lemma~\ref{l2}, one is able to characterize this in terms of Bell polynomials as (\ref{54.1}). By using (\ref{9)}), relation (\ref{54.05}) can be recovered from (\ref{54.1}).
\end{proof}
An interesting consequence of the above result is that the ratio of some carefully written nested Bell polynomials does not depend on the argument sequence, \emph{i.e.} there holds the relation:

\begin{equation}\label{B/B}
\frac{B_{n,k_1}(y'')}{B_{n,k_2}(y')}=\frac{(k_1!)^{k_2-1} } {(k_2!)^{k_1-1} }
\end{equation}
where ${y_n}'=B_{n,k_1}(x)$ and ${y_n}''=B_{n,k_2}(x)$. Identity~(\ref{B/B}) is proved by noting the relation $ (x^{*k_1})^{*k_2}=(x^{*k_2})^{*k_1} $, using Lemma~\ref{l2} and relation (\ref{9)}) subsequently.

Finally in the next theorem, it is shown that the inverse relation of a single partial Bell polynomial can be easily obtained from the convolution calculus point of view. This is realized by raising a sequence to a (convolution) power and then taking its root.
\begin{theorem}\label{tinv}
There hold the inverse relations:

\begin{equation} \label{90.1}
\left\lbrace
\begin{array}{ll}
y_n =\hat{B}_{n,k} (x)																								&               \\
x_n	=\sum _{j=0}^{n-1}      \binom{1/k}{j}        x_{1} ^{1-kj} \hat{B}_{n-1,j} (w)	& n,k\in {\mathbb{N}}
\end{array}
\right.
\end{equation}
where $w_n=y_{n+k}$, and

\begin{equation} \label{90)}
\left\lbrace
\begin{array}{ll}
y_n =					B_{n,k} (x) 			&																	\\
x_n	=\sum _{j=0}^{n-1}  (1/k)_{j}   x_{1} ^{1-kj}  nB_{n-1,j} (   w  )	& n,k\in {\mathbb{N}}
\end{array}
\right.
\end{equation}
where $w_n=y_{n+k} / \binom{n+k}{k} $, for partial Bell polynomials of ordinary and exponential types respectively.
\end{theorem}

\begin{proof}
Using (\ref{16)}) and Lemma~\ref{l6} consecutively proves the inverse relation (\ref{90.1}) for ordinary Bell polynomials. Relation \eqref{9)} can then be used to obtain the inverse relation for partial exponential Bell polynomials (\ref{90)}) from (\ref{90.1}).
\end{proof}

Theorem~\ref{tinv} can be alternatively proved in a more tedious way without using Lemmas~\ref{l2} and \ref{l6} (see appendix).

\section{Algorithms for computing Bell polynomials}\label{sec:AL}
Several methods have been proposed for computing Bell polynomials so far. For instance in \cite{Melman}, authors provide an efficient algorithm to compute Bell polynomials using compositae of generating functions. However computing a Bell polynomial with the argument sequence $x=\lbrace x_1,x_2,\cdots\rbrace$ using this method requires finding an analytic function such that its successive derivatives at some point coincides with the sequence $x$. One then needs to decompose such a function manually to elementary functions, before proceeding with the algorithm~\cite{Melman}. Unfortunately, it is not clear how such a function can be found in practice, for a given $x$. Hence we need to rely on another approach to compute Bell polynomials with arbitrary argument sequences. Using definition (\ref{3}) directly is the basic option of this kind, which is however the most computationally demanding. To see why, note that the number of monomials appearing in (\ref{3}) equals $p_k(n)$, which is the number of ways integer $n$ can be partitioned into $k$ parts. The average of these numbers for a fixed $n$ can be represented by the partition function $\sum_{k=0}^n p_k(n)/(n+1)=p(n)/(n+1)$. Allowing for a large $n$, this yields in the asymptotic relation:
$$
\frac{\sum_{k=0}^n p_k(n)}{n+1} \sim \frac{\exp(\pi\sqrt{2n/3})}{4n(n+1)\sqrt{3}}
$$
as $n\to \infty$. This exponential growth rate accounts for the number of sums only, let alone the multiplications involved in computing each monomial in (\ref{3}). As a more efficient alternative to (\ref{3}), a partial exponential Bell polynomial $B_{n,k}(x)$ is conventionally calculated using the recurrence relation (\ref{3.1}). This is also how the custom method implemented in MATLAB~\cite{matlab} works. The aim of this section is to provide new algorithms which are more efficient than (\ref{3.1}) at computing $B_{n,k}(x)$. All this algorithms take the general sequence $x=\lbrace x_1,x_2,\cdots,x_n\rbrace$ as the input.

First, it is worthwhile noting that calculating ordinary Bell polynomials first and then performing a conversion step using relation (\ref{9)}) is much more efficient than a direct use of (\ref{3.1}). Because despite (\ref{14)}), relation (\ref{3.1}) includes $n-k+1$ multiplications by binomial coefficients. In addition, a direct usage of (\ref{3.1}) requires computing the binomial coefficients, which with respect to the first index $n$, adds a cost of order $O(n^2)$ at minimum. However conducting an indirect calculation using ordinary Bell polynomials requires only $O(n)$ for conversion purposes. This explains why we base our main results in this note on the the following algorithm:
\paragraph{Algorithm 1}
\begin{equation}\label{91}
\begin{array}{ll}
\textbf{Set} \quad y_i \gets x_i/i!, \quad i=1,2, \ldots, n-k+1	\\
\textbf{Initiate} \quad \hat{B}_{0,0}(y)=1, \hat{B}_{i,0}(y)=0, \quad i=1,2, \ldots, n-k 	\\
\textbf{For} \quad l=1,2, \ldots, k \quad \textbf{do}												\\
\quad\quad \textbf{For} \quad i=l,l+1, \ldots, n-k+l \quad \textbf{do}										\\
\quad\quad \quad \quad \hat{B}_{i,l}(y) \gets \sum_{j=1}^{i-l+1} y_j \hat{B}_{i-j,l-1}(y)	\\
\quad\quad \textbf{end}																\\
\textbf{end} \\
\textbf{Set} \quad B_{n,k}(x) \gets (n!/k!)\hat{B}_{n,k}(y)
\end{array}
\end{equation}
rather than the classical approach (\ref{3.1}). Computation complexity of Algorithm~1 is given in the following theorem.

\begin{proposition}\label{t2}
Algorithm~1 needs exactly $Q^{(1)}_{n,k}$ basic arithmetic operations where

\begin{equation}\label{91.8}
Q^{(1)}_{n,k}=kn^2+2(-k^2+k+1)n+k^3-2k^2+2
\end{equation}
\end{proposition}

\begin{proof}
The conversion steps and the summation in (\ref{91}) consist of $2n-k+2$ and $2(i-l)+1$ operations respectively. Taking into account the nested loops yields

\begin{equation}
Q^{(1)}_{n,k}	=	2n-k+2+\sum_{l=1}^{k} \sum_{i=l}^{n-k+l} (2(i-l)+1)		
\end{equation}
which gives the required result.
\end{proof}

\subsection{Reduction of the second index}\label{subsec:second-index}

As Theorem~\ref{t1} suggests, a modified algorithm can be constructed on the basis of Algorithm~1 for computing partial Bell polynomial $B_{n,k}(x)$ ($n,k \geq 2$) after factorizing index $k$ into integers $p_j \in \mathbb{N}$ as

\begin{equation}\label{91.9}
k=\prod_{j=1}^{\sigma} p_j
\end{equation}
where

\begin{equation}\label{91.95}
p_1 \geq p_2 \geq \ldots \geq p_{\sigma}
\end{equation}
This can be done in the extreme case, by considering $p_j$'s to be primes. There are prime factorization schemes known to have sub-exponential complexity with respect to the number of bits, i.e. faster than $O\left((1+\epsilon)^b\right)$ where $b=\lfloor{\log_2 k}\rfloor+1$ and $\epsilon>0$ \cite{b1,prime_subex}. The proposed algorithm is described below.
\paragraph{Algorithm 2}
\begin{equation}\label{92}
\begin{array}{ll}
\textbf{Set} \quad y_i \gets x_i/i!, \quad i=1,2, \ldots, n+\sigma-\sum_{m=1}^{\sigma} p_m	\\
\textbf{For}\quad j=1, 2, \ldots, \sigma \quad\textbf{do}												\\
\quad\quad\textbf{Calculate}\quad r_j \gets \sum_{m=1}^{\sigma-j+1} p_m \\
\quad\quad\textbf{Initiate} \quad \hat{B}_{0,0}(y)=1, \hat{B}_{i,0}(y)=0, \quad i=1,2, \ldots, n-r_j + \sigma-j	\\
\quad\quad\textbf{For} \quad l=1,2, \ldots, p_{\sigma-j+1} \quad \textbf{do}												\\
\quad\quad\quad\quad \textbf{For} \quad i=l,l+1, \ldots, l+n+\sigma-j-r_j \quad \textbf{do}										\\
\quad\quad\quad\quad \quad \quad \hat{B}_{i,l}(y) \gets \sum_{m=1}^{i-l+1}y_m \hat{B}_{i-m,l-1}(y)	\\
\quad\quad\quad\quad \textbf{end}																\\
\quad\quad\textbf{end}  \\
\quad\quad \quad\quad \textbf{Set}\quad {y}_l \gets 0,	\quad l=1,  2  , \ldots,  	p_{\sigma-j+1}-1\\
\quad\quad \quad\quad \textbf{Set}\quad {y}_l \gets \hat{B}_ {l, p_{\sigma-j+1}} (y),	\quad  l=p_{\sigma-j+1},  p_{\sigma-j+1}+1  , \ldots,  n-r_j+p_{\sigma-j+1} + \sigma-j\\
\textbf{end}																								\\
\textbf{Set} \quad B_{n,k}(x) \gets n!y_n/k!
\end{array}
\end{equation}
Computation complexity of Algorithm~2 is given in the following theorem.

\begin{proposition}\label{t3}
The exact number of basic arithmetic operations demanded by Algorithm~2 (\ref{92}) is given by

\begin{equation}\label{92.01}
Q^{(2)}_{n,k}=a(k)n^2+b(k)n+c(k)
\end{equation}
where

\begin{equation}\label{92.1}
\begin{array}{l}
a(k)=\sum_{j=1}^{\sigma}p_j \\
b(k)=2+2\sum_{j=1}^{\sigma}(j-\sum_{m=1}^{j}p_m )p_j \\
c(k)= \sum_{j=1}^{\sigma}( j-\sum_{m=1}^{j}p_m )^2p_j -\sum_{j=1}^{\sigma}p_j +2\sigma
\end{array}
\end{equation}
\end{proposition}

\begin{proof}
Noting the cost of conversion steps ($2n+\sigma +1-\sum_{m=1}^{\sigma}p_m$) and the calculations in the loops of (\ref{92}) it is possible to write

\begin{equation}\label{93}
Q^{(2)}_{n,k}=2n+2\sigma-\sum_{m=1}^{\sigma}p_m +\sum_{j=1}^{\sigma}\sum_{i,l}(2(i-l)+1)
\end{equation}
where the inner summation is held over all ordered pairs $(i,l)$ satisfying $  1\leq l \leq p_{\sigma-j+1} $ and $ l \leq i \leq l+n+\sigma-j-\sum_{m=1}^{\sigma-j+1} p_m $. Expression (\ref{93}) can then be written in the form (\ref{92.01}) after some manipulations.
\end{proof}

As it can be seen in Propositions~\ref{t2} and \ref{t3}, there holds:

\begin{equation}\label{94}
\begin{array}{l}
Q^{(1)}_{n,k} \sim kn^2	\\
Q^{(2)}_{n,k} \sim ( \sum_{j=1}^{\sigma}p_j) n^2
\end{array}
\end{equation}
for computation costs of Algorithms~1 and 2, when $n \gg k$. This is a significant improvement as $\sum_{j=1}^{\sigma}p_j\leq \prod_{j=1}^{\sigma}p_j=k$ holds for all combinations of $p_j \in \mathbb{N}$ where $\sigma,p_j \geq 2$. In order to see this, define $\underline{p}=\min_{j}(p_j)$, $a_j=p_j-\underline{p}$ and $\Omega_{\sigma-j}$ as the set of all subsets of $\Omega=\lbrace 1,2, \ldots, \sigma \rbrace$ with $\sigma-j$ members. One can then write

\begin{equation}\label{95}
\begin{array}{ll}
k	&=\prod_{j=1}^{\sigma} (\underline{p}+a_j)	\\
	&=\underline{p}^\sigma + \underline{p}^{\sigma-1}(\sum_{j=1}^{\sigma}a_j)+\sum_{j=0}^{\sigma-2}\underline{p}^j \sum_{\omega \in \Omega_{\sigma-j}}\prod_{i \in \omega} a_i	\\
	&\geq \sigma\underline{p}+\sum_{j=1}^{\sigma}a_j \\
	&=\sum_{j=1}^{\sigma}p_j \\	
\end{array}
\end{equation}
where the last inequality is derived from $\underline{p}^\sigma \geq \sigma \underline{p}$ which holds for all natural numbers $\underline{p}, \sigma \geq 2$. It can be shown that the equality holds only in the case $p_1=p_2=\sigma=2$. Note that if $n \gg k$ does not apply, then Algorithm~2 does not necessarily outperform Algorithm~1 with respect to computation cost. In that case, even the cost of factorizing $k$ before running (\ref{92}) is not negligible any more.

In the following, some more details about Algorithm~2 are discussed. Firstly, it is worth remarking that the best performance is expected when integer $k$ is factorized as much as possible, i.e. into primes, before running Algorithm~2. This is obvious because when $\sigma$ is large enough, according to the inequality:

\begin{equation}\label{96}
\begin{array}{ll}
k \geq ( \min_{j}(p_j))^\sigma \geq \sigma \max_{j}(p_j) \geq \sum_{j=1}^{\sigma}p_j
\end{array}
\end{equation}
increasing $\sigma$ results in an exponentially growing gap between the two computation costs $Q^{(1)}_{n,k}$ and $Q^{(2)}_{n,k}$ asymptotically.

Secondly, the order of factors in calculation of the nested Bell polynomials

\begin{equation}\label{97}
\begin{array}{l}
y_l \gets \hat{B}_{l,p_{\sigma}}(y)	\\
y_l \gets \hat{B}_{l,p_{\sigma-1}}(y)	\\
\vdots \\
y_l \gets \hat{B}_{l,p_{1}}(y)
\end{array}
\end{equation}
affects the algorithm performance. This order is chosen by default as (\ref{91.95}) in (\ref{92}) for Algorithm~2. Ordering factors in this way is the best choice with regard to computation cost, as shown in the following theorem.

\begin{theorem}\label{thm:order}
The ordering (\ref{91.95}) is optimal in Algorithm~2.
\end{theorem}
\begin{proof}
Assume $p_{i} \geq p_{i+1}$ and consider the two different factorization orderings
$$p_1,p_2, \ldots, p_{i-1}, p_i, p_{i+1}, \ldots, p_{\sigma} $$
and
$$p_1,p_2, \ldots, p_{i-1}, p_{i+1}, p_i, \ldots, p_{\sigma} $$
in Algorithm~2 with the corresponding computation costs of $Q^{(2)}_{n,k}$ and $\hat{Q}^{(2)}_{n,k}$ respectively. For simplicity, define $u=i-1-\sum_{j=1}^{i-1} p_j$. Using relation (\ref{92.01}), it is possible to write

\begin{equation}\label{98}
\begin{array}{l}
					\hat{Q}^{(2)}_{n,k}-Q^{(2)}_{n,k}=	\\
					2n \left( \left(u+1-p_{i+1}\right)p_{i+1}-\left(u+1-p_i\right)p_{i}+\left(u+2-p_i-p_{i+1}\right)(p_i-p_{i+1}) \right)+	\\
					\left( (u+1-p_{i+1})^2p_{i+1}-(u+1-p_i)^2p_{i}+(u+2-p_i-p_{i+1})^2(p_i-p_{i+1}) \right)=	\\
					2n \left( (u+1)\left(p_{i+1}-p_i\right)+\left(p_i-p_{i+1}\right)\left(p_i+p_{i+1}\right)+\left(u+2-p_i-p_{i+1}\right)\left(p_i-p_{i+1}\right) \right) +	\\
					p^3_{i+1}-p^3_{i}+2(u+1)(p_{i}-p_{i+1})(p_{i}+p_{i+1})-(u+1)^2(p_{i}-p_{i+1})+ \\
					(u+2-p_{i}-p_{i+1})^2(p_{i}-p_{i+1})=\\
					(p_i-p_{i+1}) \left( 2n+ p_{i}p_{i+1}-2p_{i}-2p_{i+1}+2u+3  \right)
\end{array}
\end{equation}
On the other hand since $p_j \geq 2$ holds for all $1\leq j \leq \sigma$ one can write

\begin{equation}\label{99}
\begin{array}{rl}
p_{i}+p_{i+1}-u&\leq	\sum_{j=1}^{\sigma}p_j-\sigma+2\\
				&\leq	k-\sigma+2
\end{array}
\end{equation}
where the last inequality was derived from (\ref{95}). Using the upper bound (\ref{99}) in (\ref{98}) results in

\begin{equation}\label{100}
\hat{Q}^{(2)}_{n,k}-Q^{(2)}_{n,k}	\geq (p_i-p_{i+1}) \left( 2(n-k)+ p_{i}p_{i+1}+2(\sigma-2) +3 \right)\geq 0
\end{equation}
This proves $\hat{Q}^{(2)}_{n,k} \geq Q^{(2)}_{n,k}$ and therefore justifies the ordering (\ref{91.95}) as being optimal.
\end{proof}

\subsection{Reduction of the first index}\label{subsec:first-index}
As it was discussed above, the second index in a Bell polynomial calculation may only be reduced when it is not prime. Analogously, the first index cannot be reduced in the general case. More precisely, a reduction of the Bell polynomial's first index is only possible for the argument sequences with zero initial samples. For these arguments, the computation cost may be further improved by reducing the first index in the Bell polynomial using the following lemma.

\begin{lemma}\label{nthm}
Define $y_n=x_{n+n_0}/(n+n_0)_{n_0}$, where $x$ is a sequence such that $x_n=0$ for $n \leq n_0$ where $n_0 \geq 0$. There holds:

\begin{equation}\label{nbreakExp}
B_{n,k}(x)=(n)_{kn_0}B_{n-kn_0,k}(y)
\end{equation}
\end{lemma}

\begin{proof}
First we introduce the auxiliary sequences $\hat{x}_n=x_n/n!$ and $\hat{y}_n=y_n/n!$. To prove this theorem, one should then simply exploit the shift invariance property of convolution and write $\hat{x}^{*k}_{n}=\hat{y}^{*k}_{n-kn_0}$ where $\hat{y}_n=\hat{x}_{n+n_0}$ and $\hat{x}_n=0$ for $n \leq n_0$. Using Lemma~\ref{l2}, This results in the identity

\begin{equation}\label{nbreak}
\hat{B}_{n,k}(\hat{x})=\hat{B}_{n-kn_0,k}(\hat{y})
\end{equation}
for ordinary Bell polynomials. Finally, Using the relation (\ref{9)}) leads to the similar identity (\ref{nbreakExp}) for exponential Bell polynomials as given in the statement of this theorem.
\end{proof}

If the argument sequence of Bell polynomial contains $n_0$ initial zeros, then some unnecessary calculations can be avoided based on Lemma~\ref{nthm}. In this case instead of using Algorithm~2 to calculate $B_{n,k}(x)$, one can opt to compute the right hand side of (\ref{nbreakExp}) using the same technique. This results in the most general and the most efficient algorithm, which exploits reductions in both indexes. Suppose that $x$ is a sequence such that $x_n=0$ for $n \leq n_0$ where $n_0 \geq 0$ and the second index $k$ can be factorized as (\ref{91.9}). Then Bell polynomial $B_{n,k}(x)$ can be computed using the following algorithm:
\paragraph{Algorithm 3}
\begin{equation}\label{GenAl}
\begin{array}{ll}
\textbf{Set} \quad y_i \gets x_{i+n_0}/(i+n_0)!, \quad i=1,2, \ldots, n-kn_0+\sigma-\sum_{m=1}^{\sigma} p_m	\\
\textbf{For}\quad j=1, 2, \ldots, \sigma \quad\textbf{do}												\\
\quad\quad\textbf{Calculate}\quad r_j \gets \sum_{m=1}^{\sigma-j+1} p_m \\
\quad\quad\textbf{Initiate} \quad \hat{B}_{0,0}(y)=1, \hat{B}_{i,0}(y)=0, \quad i=1,2, \ldots, n-kn_0-r_j + \sigma-j	\\
\quad\quad\textbf{For} \quad l=1,2, \ldots, p_{\sigma-j+1} \quad \textbf{do}												\\
\quad\quad\quad\quad \textbf{For} \quad i=l,l+1, \ldots, l+n-kn_0+\sigma-j-r_j \quad \textbf{do}										\\
\quad\quad\quad\quad \quad \quad \hat{B}_{i,l}(y) \gets \sum_{m=1}^{i-l+1}y_m \hat{B}_{i-m,l-1}(y)	\\
\quad\quad\quad\quad \textbf{end}																\\
\quad\quad\textbf{end}  \\
\quad\quad \quad\quad \textbf{Set}\quad {y}_l \gets 0,	\quad l=1,  2  , \ldots,  	p_{\sigma-j+1}-1\\
\quad\quad \quad\quad \textbf{Set}\quad {y}_l \gets \hat{B}_ {l, p_{\sigma-j+1}} (y),\\
\quad\quad \quad\quad l=p_{\sigma-j+1},  p_{\sigma-j+1}+1  , \ldots,  n-kn_0-r_j+p_{\sigma-j+1} + \sigma-j\\
\textbf{end}																								\\
\textbf{Set} \quad B_{n,k}(x) \gets n!y_{n-kn_0}/k!
\end{array}
\end{equation}

In the above algorithm, $y_l=0$ is assumed to hold for all $l\leq 0$. Therefore, the last line in (\ref{GenAl}) indicates that $B_{n,k}(x)=0$ holds for all $x$ when $n\leq kn_0$. In this trivial case, $B_{n,k}(x)$ is immediately obtained and there is no need to run Algorithm~3. Otherwise, let us denote the number of basic arithmetic operations needed in Algorithm~3~(\ref{GenAl}) to compute the Bell polynomial $B_{n,k}(x)$ by the generalized notation $Q^{(3)}_{n,k,n_0}$. There holds:
$$
Q^{(3)}_{n,k,0}=Q^{(2)}_{n,k}
$$
which means that no improvement is possible over Algorithm~2 when $n_0=0$. In contrast, as stated in the following proposition, the computation complexity can be significantly improved in the general case, \emph{i.e.} when $n_0\geq 0$.

\begin{proposition}
Algorithm~3 needs $Q^{(3)}_{n,k,n_0}$ basic arithmetic operations to compute a Bell polynomial $B_{n,k}(x)$, where
\begin{equation}\label{101}
Q^{(3)}_{n,k,n_0}=Q^{(2)}_{n-kn_0,k}+ \max \left\lbrace 0 , n_0+\sigma - \sum_{m=1}^{\sigma} p_m \right\rbrace
\end{equation}
and $Q^{(2)}$ is defined in (\ref{92.01}).
\end{proposition}

Due to (\ref{101}), the number of basic arithmetic operations avoided in Algorithm~3 compared to a direct use of Algorithm 2 equals

\begin{equation}\label{102}
2a(k)n_0k n
-a(k)k^2n_0^2+b(k)kn_0
- \max \left\lbrace 0 , n_0+\sigma - \sum_{m=1}^{\sigma} p_m \right\rbrace
\end{equation}
where all the terms except the first one are independent of $n$. This shows a reduction equivalent to an order of magnitude with respect to $n$, when the argument sequence $x$ has at least one zero sample initially. This is significant when $n$ is large.

\subsection{Numerical evaluation}\label{subsec:numerical}
In this section the presented algorithms are evaluated in numerical examples. First, a simple experiment was conducted to show the benefits of index reduction techniques on Bell polynomials computation cost. In this experiment, the speed-up ratio
\begin{equation}\label{R_n,k}
R_{n,k}=\frac{Q^{(1)}_{n,k}}{Q^{(2)}_{n,k}}
\end{equation}
between Algorithm~2 and Algorithm~1 which does not exploit index reductions are evaluated for computing Bell polynomials with various indexes $n$ and $k$ and general argument sequences with no information available about their initial samples ($n_0=0$). The result is shown in Figure~\ref{fspeedup}. By increasing the value of $n$, each curve in Figure~\ref{fspeedup} tends to the limit ratio
$$
\lim_{n\to +\infty} R_{n,k}=\frac{\prod_{j=1}^{\sigma} p_j}{\sum_{j=1}^{\sigma} p_j}
$$
which is always greater than one. This indicates that a Bell polynomial with a large enough $n$ is always cheaper to compute using Algorithm~2 which uses index reductions, than Algorithm~1. Note that Algorithm~1 is always more efficient than the conventional recurrence relation (\ref{3.1}) and the computation complexity becomes significant only when the polynomial indexes are large.

\begin{figure}
\centering
\includegraphics[scale=0.5]{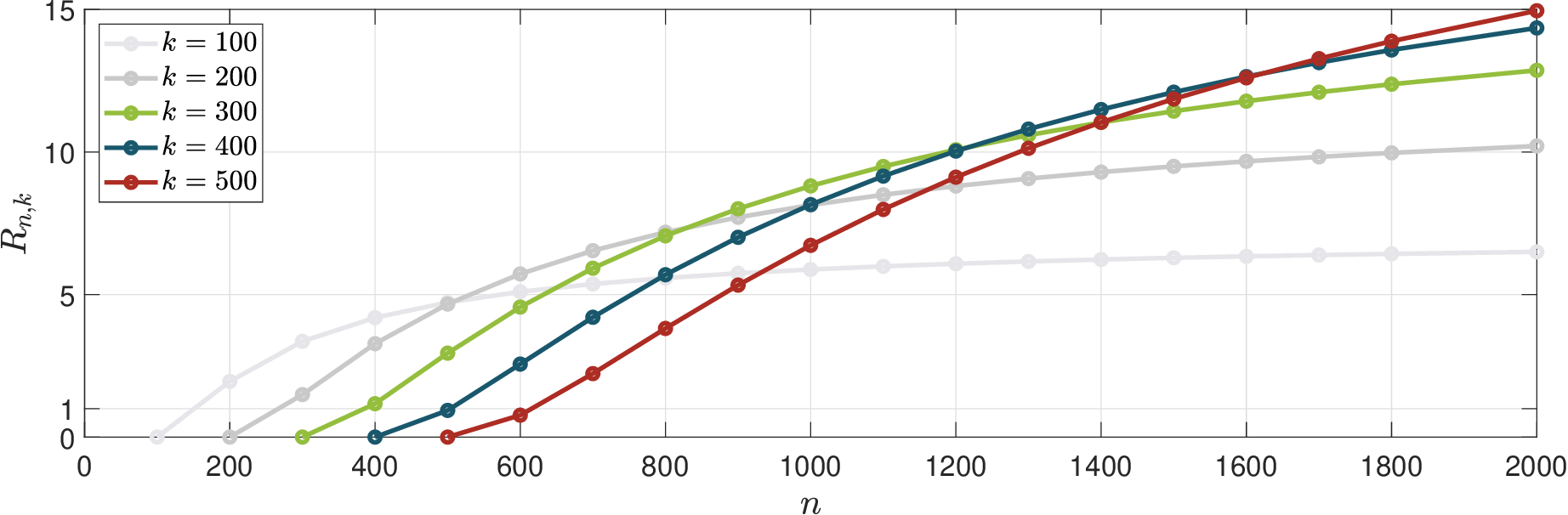}
\caption{\footnotesize{The speed-up ratios (\ref{R_n,k}) between the two Algorithms~1 and 2 for computing Bell polynomials with various indexes $n$ and $k$.}}
\label{fspeedup}
\end{figure}

Next, we compare all the computation methods in this section in a unified experiment. For this purpose, we will use the recurrence relation (\ref{3.1}) and Algorithms~1, 2 and 3 for computing a Bell polynomial with a general argument $x$. In this experiment, $n_0$ number of initial samples in the argument sequence are assumed to be zero, \emph{i.e.} $x_1=x_2=\cdots=x_{n_0}=0$. The time needed for this computation is recorded for each algorithm in Table~\ref{tb1}. The algorithms were run on a computer with an Intel(R) Core(TM) i5-4310U CPU at 2.00GHz 2.6GHz and 8 GB of memory.

\begin{table}
\begin{center}
\begin{tabular}{|l c c c c c|}
\hline
\multicolumn{6}{|c|}{Computation time (s)}\\
 						& 		& $n=20$, $k=6$ & $n=20$, $k=8$ & $n=20$, $k=10$	& $n=20$, $k=12$\\  
\hline
Recurrence (\ref{3.1})	&\vline &7.3791			&13.6547		&49.2792			&20.1058\\         
Algorithm~1				&\vline &6.9059			&12.7310		&47.8660			&20.0026\\
Algorithm~2				&\vline &5.5574			&7.0280 		&14.2892			&7.3457\\
Algorithm~3 ($n_0=1$)	&\vline &3.4950   		&2.8674   		&2.4046				&1.5961\\
Algorithm~3 ($n_0=2$)	&\vline &1.3917   		&0.4704   		&NA					&NA\\
Algorithm~3 ($n_0=3$)	&\vline &0.0370   		&NA   			&NA					&NA\\
\hline
\hline
\end{tabular}
\end{center}
\caption{\label{tb1} \footnotesize{Times needed to compute $B_{n,k}(x)$ in seconds.}}
\end{table}

The general superiority of Algorithms~1, 2 and 3 to the classical approach, that is the recurrence relation (\ref{3.1}), is emphasized in Table~\ref{tb1} in this experiment. As expected, Algorithm~3 is the fastest since it exploits the knowledge that $x$ has a few zero samples and therefore avoids some unnecessary calculations. When more initial samples of $x$ are known to be zero, more unnecessary calculations can be prevented which results in yet shorter computation times. With a too large $n_0$, \emph{i.e.} when $n\leq kn_0$, the computation task becomes trivial and $B_{n,k}(x)=0$ holds for all $x$ which explains the entries 'NA' in Table~\ref{tb1}. In the most general case when no samples in $x$ are known to be zero, Algorithm~3 is reduced to Algorithm~2 which only enjoys reductions in the second index. This algorithm still outperforms both Algorithm~1 and the recurrence relation (\ref{3.1}) according to Table~1. Finally as it was expected, all the presented algorithms in this section including Algorithm~1 exhibit shorter computation times than the conventional approach. 

\section{Conclusion}\label{sec:conc}
In this note the connection of Bell polynomials to convolution calculus was addressed. The proofs of several properties of Bell polynomials can be shortened considerably from this convolution calculus point of view. For instance, it was shown how ordinary Bell polynomials can be used to obtain compact representations for rational convolution powers of sequences and some identities involving Bell polynomials were introduced. It was later shown that these identities have important consequences in computational sense.

In Section~\ref{sec:AL}, three algorithms were introduced for computing partial exponential Bell polynomials, which are more efficient than the classical approach (\ref{3.1}) used for computing these polynomials. It was first addressed how an indirect calculation based on ordinary Bell polynomials (Algorithm~1) is computationally more efficient than the conventional approach where exponential Bell polynomials are computed recursively. Then Algorithm~2 was introduced which reduces the second index $k$ of a Bell polynomial in computation. Although there are cases in which using Algorithm~1 is cheaper, Algorithm~2 was a significant improvement over Algorithm~1 in most troublesome cases, \emph{e.g.} when the first index $n$ is large. Finally Algorithm~3 was presented as an improvement of Algorithm~2, by incorporating reductions in both indexes of a Bell polynomial. This algorithm can exploit more information of the argument sequence through parameter $n_0$ than the previous algorithms to further reduce the computation cost. Superiority of Algorithm~3 is accentuated when the first index of Bell polynomial $n$ grows larger than its second index $k$. However as it was pointed out earlier, sometimes Algorithm~3 is not superior when $n$ is too close to $k$ and the argument sequence does not have any zero initial samples. In this case, it may demand more arithmetic operations than Algorithm~1 for computation. Algorithms~1, 2 and 3 can make the use of Bell polynomials in scientific computations much less costly.\\

\subsection{Acknowledgements}
I would like to thank the anonymous reviewers for their insight and helpful comments.

\subsection{Funding and conflicts of interests}
This research was funded by KTH Royal Institute of Technology. The Author declares that there is no conflict of interest.

\section{Appendix}
\noindent\textbf{Alternative proof of Theorem~\ref{tinv}}\\
\noindent Using the generating function of partial exponential Bell polynomials we may write \cite[Theorem 11.1]{10}
$$
\sum_{n=k}^{+\infty}y_n \frac{z^n}{n!}=\frac{1}{k!}\left(\sum_{n=1}^{+\infty}x_n\frac{z^n}{n!}\right)^k
$$
and thereby using an appropriate variable change we obtain
\begin{align*}
\sum_{n=1}^{+\infty}x_n\frac{z^n}{n!}&=\left(k!\sum_{n=0}^{+\infty}y_{n+k} \frac{z^{n+k}}{(n+k)!}\right)^{1/k} \\
									&=z\left(k!\sum_{n=0}^{+\infty}y_{n+k} \frac{z^n}{(n+k)!}\right)^{1/k}\\
									&=z\left(\sum_{n=0}^{+\infty}\frac{y_{n+k}}{\binom{n+k}{k}} \frac{z^n}{n!}\right)^{1/k}\\
									&=zy_k^{1/k}\left(\sum_{n=0}^{+\infty} \frac{w_n}{y_k}\frac{z^n}{n!}\right)^{1/k}\\
									&=zy_k^{1/k} f(z)
\end{align*}
Then by calculating the successive derivatives of $f(z)$ based on Theorem~B of \cite[section 3.5]{comtet} and using the relation of potential polynomials with Bell polynomials thereafter, we arrive at
\begin{align}\label{eqn:appen:invrel}
\sum_{n=1}^{+\infty}x_n\frac{z^n}{n!}&=zy_k^{1/k}\left(\sum_{n=0}^{+\infty}\frac{z^n}{n!} \sum_{j=0}^{n}(1/k)_j B_{n,j}(w/y_k) \right)\nonumber\\
									&=\sum_{n=0}^{+\infty}\frac{z^{n+1}}{n!} \sum_{j=0}^{n}y_k^{1/k-j}(1/k)_j B_{n,j}(w)\nonumber\\
									&=\sum_{n=1}^{+\infty}\frac{n z^n}{n!} \sum_{j=0}^{n-1}y_k^{1/k-j}(1/k)_j B_{n-1,j}(w)
\end{align}
Since
$$y_k^{1/k-j}=\left(B_{k,k}(x)\right)^{1/k-j}=x_1^{1-kj}$$
by equating the corresponding terms on both sides of (\ref{eqn:appen:invrel}) we obtain (\ref{90)}). Relation \eqref{9)} can then be used to obtain the inverse relation for partial ordinary Bell polynomials (\ref{90.1}) form (\ref{90)}).
\hfill$\square$

\end{document}